\documentclass[oneside,reqno,english]{amsart}
\usepackage{lmodern}
\usepackage[T1]{fontenc}
\usepackage[latin9]{inputenc}
\usepackage{babel}
\usepackage{varioref}
\usepackage{float}
\usepackage{amstext}
\usepackage{amsthm}
\usepackage{amssymb}
\usepackage{graphicx}
\usepackage[unicode=true,pdfusetitle,
 bookmarks=true,bookmarksnumbered=false,bookmarksopen=false,
 breaklinks=false,pdfborder={0 0 1},backref=false,colorlinks=false]
 {hyperref}

\makeatletter

\providecommand{\tabularnewline}{\\}
\floatstyle{ruled}
\newfloat{algorithm}{tbp}{loa}
\providecommand{\algorithmname}{Algorithm}
\floatname{algorithm}{\protect\algorithmname}

\numberwithin{equation}{section}
\numberwithin{figure}{section}
\newenvironment{lyxcode}
	{\par\begin{list}{}{
		\setlength{\rightmargin}{\leftmargin}
		\setlength{\listparindent}{0pt}
		\raggedright
		\setlength{\itemsep}{0pt}
		\setlength{\parsep}{0pt}
		\normalfont\ttfamily}%
	 \item[]}
	{\end{list}}
\theoremstyle{plain}
\newtheorem{thm}{\protect\theoremname}[section]
\theoremstyle{remark}
\newtheorem{rem}[thm]{\protect\remarkname}
\theoremstyle{plain}
\newtheorem{prop}[thm]{\protect\propositionname}
\theoremstyle{plain}
\newtheorem{lyxalgorithm}[thm]{\protect\algorithmname}
\theoremstyle{remark}
\newtheorem{claim}[thm]{\protect\claimname}

\newcommand{\Val}{\mbox{\rm Val}}
\newcommand{\outdeg}{\mbox{\rm outdeg}}

\makeatother

\providecommand{\algorithmname}{Algorithm}
\providecommand{\claimname}{Claim}
\providecommand{\propositionname}{Proposition}
\providecommand{\remarkname}{Remark}
\providecommand{\theoremname}{Theorem}

\begin{document}
\begin{lyxcode}
\title[Optimization with directed unreliable edges]{A dual ascent algorithm for asynchronous distributed optimization  with unreliable directed communications} 
\end{lyxcode}

\begin{abstract}
We show that the averaged consensus algorithm on directed graphs with
unreliable communications by Bof-Carli-Schenato has a dual optimization
interpretation, which could be extended to the case of distributed
optimization. We report on our numerical simulations for the distributed
optimization algorithm for smooth and nonsmooth functions. 
\end{abstract}

\author{C.H. Jeffrey Pang}
\thanks{C.H.J. Pang acknowledges grant R-146-000-265-114 from the Faculty
of Science, National University of Singapore. }
\curraddr{Department of Mathematics\\ 
National University of Singapore\\ 
Block S17 08-11\\ 
10 Lower Kent Ridge Road\\ 
Singapore 119076 }
\email{matpchj@nus.edu.sg}
\date{\today{}}

\maketitle
\tableofcontents{}

\section{Introduction}

Let $G=(V,E)$ be a directed graph. Consider the distributed optimization
problem
\begin{equation}
\min_{x\in\mathbb{R}^{m}}\sum_{i\in V}\left[f_{i}(x)+\frac{1}{2}\|x-\bar{x}_{i}\|^{2}\right].\label{eq:dist_opt_pblm}
\end{equation}
Here, $f_{i}(\cdot)$ are closed convex functions. The challenge in
distributed optimization is that the communications in the algorithm
need to obey the directed edges in the underlying graph. Note that
if $f_{i}(\cdot)$ are the zero functions and $m=1$, then the minimizer
of \eqref{eq:dist_opt_pblm} is exactly $\frac{1}{|V|}\sum_{i\in V}\bar{x}_{i}$,
which is precisely the distributed averaged consensus problem.

A distributed asynchronous algorithm for averaged consensus on a directed
graph with unreliable communications was designed in \cite{Bof_Carli_Schenato_2017}.
The paper \cite{Bof_Carli_Schenato_2017} was inspired by two algorithms
for averaged consensus in the literature. In an asynchronous setting,
\cite{Benezit_Blondel_Thiran_Tsitsilkis_Vetterli_2010_weighted_avr}
introduced an algorithm that reaches averaged consensus using the
so-called ratio consensus. The paper in \cite{Vaidya_Hadji_Domin_2011}
gave the idea of mass transfer used in \cite{Bof_Carli_Schenato_2017}.
(Other papers also mentioned \cite{Hadji_Vaidya_Domin_2016_running_sums}.)
The paper \cite{Bof_Carli_Schenato_2017} also proved linear convergence
of their algorithm, and pointed out algorithms in \cite{VZCP_Schenato_2016,Cattivelli_Sayed_2010,Domin_Hadji_2010_smartgrid}
need the averaged consensus algorithm and its linear convergence as
a building block. The work in \cite{Bof_Carli_Schenato_2017} has
led to other strategies for distributed asynchronous optimization
on directed and unreliable communications \cite{Notarstefano_gang_Newton_2017,Scutari_ASY_SONATA_2018}. 

Another idea in this paper comes from our related work on solving
the distributed problem \eqref{eq:dist_opt_pblm} on undirected graphs.
In \cite{Pang_Dist_Dyk,Pang_sub_Dyk,Pang_rate_D_Dyk,Pang_level_sets_Dyk},
we proposed a distributed asynchronous optimization algorithm. The
idea behind those papers is that the problem \eqref{eq:dist_opt_pblm}
can be written as a variant of the product space formulation, and
subsequently solved with Dykstra's algorithm \cite{Dykstra83}. Dykstra's
algorithm is identical to block coordinate minimization on the dual
\cite{Han88}, and is notable because the convergence to its primal
minimizer does not rely on the existence of dual optimizers \cite{BD86,Gaffke_Mathar}.
We were also motivated by these works, as well as \cite{Hundal-Deutsch-97}
for the asynchronous operation of the algorithm. Some interesting
properties of the algorithm in \cite{Pang_Dist_Dyk,Pang_sub_Dyk,Pang_rate_D_Dyk,Pang_level_sets_Dyk}
include: being able to work on time-varying graphs, allow for partial
communication of data, allow for more than two of the $f_{i}(\cdot)$
to be indicator functions of closed convex sets (instead of being
smooth functions), has deterministic convergence with rates mostly
compatible with well known first order methods, and convergence to
the primal solution even when there are no dual optimizers.

The case of distributed optimization where $f_{i}(\cdot)$ need not
be smooth functions is also interesting in its own right; The recent
paper \cite{Borkar_distrib_dyk} worked on the case where $f_{i}(\cdot)$
are indicator functions of closed convex sets, and highlighted \cite{Aybat_Hamedani_2016,LeeNedich2013,Nedic_Ram_Veeravalli_2010,Ozdaglar_Nedich_Parrilo}.
Ideally, one would want to solve the problem where the quadratic regularizers
in \eqref{eq:dist_opt_pblm} were removed, but the quadratic regularizer
is needed for the algorithm in \cite{Pang_Dist_Dyk} and the follow
up papers to work.

\subsection{Contributions of this paper}

In this paper, we propose a distributed algorithm on directed graphs
with unreliable communications for the regularized optimization problem
\eqref{eq:dist_opt_pblm} which generalizes the algorithm in \cite{Bof_Carli_Schenato_2017}
(for the averaged consensus problem on direced graphs and unreliable
edges) and \cite{Pang_Dist_Dyk} (for a distributed optimization algorithm
for \eqref{eq:dist_opt_pblm} on undirected graphs). We show that
the dual objective value of \eqref{eq:dist_opt_pblm} gives a potential
function (or Lyapunov function) similar to that of \cite{Pang_Dist_Dyk}
whose value is monotonically nonincreasing throughout the algorithm. 

\section{Algorithm derivation and description}

In this section, we derive our algorithm. Those familiar with \cite{Bof_Carli_Schenato_2017}
would recognize operations $A$ and $B$ in Algorithm \ref{alg:undir-alg},
but operation $C$ there requires some preparation in the dual formulation. 

Let $\bar{m}=\frac{1}{|V|}\sum_{i\in V}\bar{x}_{i}$. We have 
\begin{equation}
\sum_{i\in V}\frac{1}{2}\|x-\bar{x}_{i}\|^{2}=\frac{|V|}{2}\|x-\bar{m}\|^{2}+\underbrace{\sum_{i\in V}\bar{x}_{i}^{T}\bar{x}_{i}-|V|\bar{m}^{T}\bar{m}}_{C}.\label{eq:make-C}
\end{equation}
So we can assume that all $\bar{x}_{i}$ in \eqref{eq:dist_opt_pblm}
are equal to $\bar{m}$. (This does not mean that a starting primal
variable needs to be $\bar{m}$.) Let $\{s_{\alpha}\}_{\alpha\in V\cup E\cup\{r\}}$
be such that 
\begin{equation}
\sum_{\alpha\in V\cup E\cup\{r\}}s_{\alpha}=|V|,\text{ and }s_{\alpha}\begin{cases}
>0 & \text{ for all }\alpha\in V\\
\geq0 & \text{ for all }\alpha\in E\cup\{r\}.
\end{cases}\label{eq:condn-on-s}
\end{equation}
Let $\mathbf{x}\in[\mathbb{R}^{m}]^{|V\cup E\cup\{r\}|}$, and for
all $i\in V$, let $\mathbf{f}_{i}:[\mathbb{R}^{m}]^{|V\cup E\cup\{r\}|}\to\mathbb{R}\cup\{\infty\}$
be defined as $\mathbf{f}_{i}(\mathbf{x})=f_{i}([\mathbf{x}]_{i})$.
Let the set $F$ be 
\[
F:=\big\{\{i,(i,j)\}:(i,j)\in E\big\}\cup\big\{\{j,(i,j)\}:(i,j)\in E\big\}\cup\big\{\{r,\alpha\}:\alpha\in V\cup E\big\}.
\]
and let the hyperplane $H_{\{\alpha_{1},\alpha_{2}\}}$, where $\{\alpha_{1},\alpha_{2}\}\in F$,
be defined by 
\[
H_{\{\alpha_{1},\alpha_{2}\}}:=\{\mathbf{x}\in[\mathbb{R}^{m}]^{|V\cup E\cup\{r\}|}:x_{\alpha_{1}}=x_{\alpha_{2}}\}.
\]
We assume the underlying graph is strongly connected, so 
\[
\cap_{\beta\in F}H_{\beta}=\big\{\mathbf{x}\in[\mathbb{R}^{m}]^{|V\cup E\cup\{r\}|}:x_{\alpha_{1}}=x_{\alpha_{2}}\text{ for all }\alpha_{1},\alpha_{2}\in V\cup E\cup\{r\}\big\}.
\]
The primal problem \eqref{eq:dist_opt_pblm} can then be equivalently
written in the product space formulation as 
\begin{equation}
\min_{\mathbf{x}\in[\mathbb{R}^{m}]^{|V\cup E\cup\{r\}|}}\sum_{\alpha\in V\cup E\cup\{r\}}\frac{s_{\alpha}}{2}\|[\mathbf{x}]_{\alpha}-\bar{m}\|^{2}+\sum_{i\in V}\mathbf{f}_{i}(\mathbf{x})+\sum_{\beta\in F}\delta_{H_{\beta}}(\mathbf{x})+C,\label{eq:2nd-primal}
\end{equation}
where $C$ is as marked in \eqref{eq:make-C}. If $\mathbf{x}^{*}\in[\mathbb{R}^{m}]^{|V\cup E\cup\{r\}|}$
is an optimal solution of \eqref{eq:2nd-primal}, then all $|V\cup E\cup\{r\}|$
components (in $\mathbb{R}^{m}$) of $\mathbf{x}^{*}$ are equal,
and are optimal solutions of \eqref{eq:dist_opt_pblm}. The (Fenchel)
dual of \eqref{eq:2nd-primal} can be calculated to be 
\begin{equation}
\sup_{{\mathbf{z}_{\alpha}\in[\mathbb{R}^{m}]^{|V\cup E\cup\{r\}|}\atop \alpha\in V\cup F}}\frac{|V|}{2}\|\bar{m}\|^{2}-\sum_{i\in V}\mathbf{f}_{i}^{*}(\mathbf{z}_{i})-\sum_{\beta\in F}\delta_{H_{\beta}^{\perp}}(\mathbf{z}_{\beta})-\sum_{\alpha\in V\cup E\cup\{r\}}\frac{s_{\alpha}}{2}\left\Vert \bar{m}-\frac{1}{s_{\alpha}}\left[\sum_{\alpha_{2}\in V\cup F}\mathbf{z}_{\alpha_{2}}\right]_{\alpha}\right\Vert ^{2}+C.\label{eq:dual-1}
\end{equation}
The case when $s_{\alpha}=1$ for all $\alpha\in V$ and $s_{\alpha}=0$
for all $\alpha\in E$ has been discussed in detail in \cite{Pang_Dist_Dyk,Pang_sub_Dyk,Pang_rate_D_Dyk,Pang_level_sets_Dyk}.
The treatment there (which traces to the original work in \cite{BD86})
implies that there is strong duality between \eqref{eq:2nd-primal}
and \eqref{eq:dual-1}, even if dual optimizers may not exist. For
convenience, instead of considering \eqref{eq:dual-1}, we consider
\begin{equation}
\begin{array}{c}
\underset{{\mathbf{z}_{\alpha}\in[\mathbb{R}^{m}]^{|V\cup E\cup\{r\}|}\atop \alpha\in V\cup F}}{\inf}\underset{i\in V}{\sum}\mathbf{f}_{i}^{*}(\mathbf{z}_{i})+\underset{\beta\in F}{\sum}\delta_{H_{\beta}^{\perp}}(\mathbf{z}_{\beta})+\underset{\alpha\in V\cup E}{\sum}\frac{s_{\alpha}}{2}\left\Vert \bar{m}-\frac{1}{s_{\alpha}}\left[\underset{\alpha_{2}\in V\cup F}{\sum}\mathbf{z}_{\alpha_{2}}\right]_{\alpha}\right\Vert ^{2}.\end{array}\label{eq:dual-2}
\end{equation}

\begin{rem}
\label{rem:On-index-r}(On the index $r$) Notice that $s_{r}$ and
$y_{r}$ remain as zero throughout Algorithm \ref{alg:undir-alg},
and $\mathbf{z}_{\{r,\alpha\}}$ also remains as zero for all $\alpha\in V\cup E$
as well. We introduced this additional index $r$ in order to simplify
the convergence proof in Section \ref{sec:Convergence-analysis}.
\end{rem}

We have the following properties:
\begin{prop}
\label{prop:sparsity}(Sparsity) The following results hold:
\begin{enumerate}
\item If $i\in V$, then $\mathbf{z}_{i}\in[\mathbb{R}^{m}]^{|V\cup E\cup\{r\}|}$
is such that $[\mathbf{z}_{i}]_{\alpha}=0$ for all $\alpha\in[V\cup E\cup\{r\}]\backslash\{i\}$. 
\item If $\{\alpha_{1},\alpha_{2}\}\in F$, then $\mathbf{z}_{\{\alpha_{1},\alpha_{2}\}}\in[\mathbb{R}^{m}]^{|V\cup E\cup\{r\}|}$
is such that $[\mathbf{z}_{\{\alpha_{1},\alpha_{2}\}}]_{\alpha}=0$
for all $\alpha\in[V\cup E\cup\{r\}]\backslash\{\alpha_{1},\alpha_{2}\}$,
and $[\mathbf{z}_{\{\alpha_{1},\alpha_{2}\}}]_{\alpha_{1}}+[\mathbf{z}_{\{\alpha_{1},\alpha_{2}\}}]_{\alpha_{2}}=0$. 
\end{enumerate}
\end{prop}

\begin{proof}
The proof is elementary and exactly the same as that in \cite{Pang_Dist_Dyk}.
(Part (1) makes use of the fact that $\mathbf{f}_{i}(\cdot)$ depends
on only the $i$-th coordinate of the input, while part (2) makes
use of the fact that $\delta_{H_{\{\alpha_{1},\alpha_{2}\}}}^{*}(\cdot)=\delta_{H_{\{\alpha_{1},\alpha_{2}\}}^{\perp}}(\cdot)$,
and $\delta_{H_{\{\alpha_{1},\alpha_{2}\}}^{\perp}}(\mathbf{z}_{\{\alpha_{1},\alpha_{2}\}})<\infty$
implies the conclusions in (2).)
\end{proof}
We now describe Algorithm \vref{alg:undir-alg}. In order to link
Algorithm \ref{alg:undir-alg} with the dual objective function \eqref{eq:dual-2},
we define\begin{subequations}\label{eq_m:y-s-x}
\begin{eqnarray}
y_{(i,j)} & := & \sigma_{i,y}-\rho_{(i,j),y}\text{ for all }(i,j)\in E\label{eq:y-s-x-y}\\
s_{(i,j)} & := & \sigma_{i,s}-\rho_{(i,j),s}\text{ for all }(i,j)\in E\label{eq:y-s-x-s}\\
x_{\alpha} & := & y_{\alpha}/s_{\alpha}\text{ for all }\alpha\in V\cup E\cup\{r\}.\label{eq:y-s-x-x}
\end{eqnarray}
\end{subequations}As we have seen in \cite{Bof_Carli_Schenato_2017},
the data $\sigma_{i,y}$ and $\sigma_{i,s}$ represent data transmitted
by node $i$, and $\rho_{(i,j),y}$ and $\rho_{(i,j),s}$ represent
data from node $i$ that has been received by node $j$ through the
edge $(i,j)$. So $\sigma_{i,y}-\rho_{(i,j),y}$ and $\sigma_{i,s}-\rho_{(i,j),s}$
represent data that have been transmitted by node $i$ to node $j$
along edge $(i,j)$ that have not yet been received by node $j$.
Figure \ref{fig:The-fig} illustrates Operations $A$ and $B$ of
the algorithm in \cite{Bof_Carli_Schenato_2017}. Hence using $y_{(i,j)}$
and $s_{(i,j)}$ to represent these data is natural. It is clear that
if $s_{\alpha}=0$, then $y_{\alpha}=0$. In such a case, the choice
of $x_{\alpha}$ is irrelevant. We want $\{\mathbf{z}_{\alpha}\}_{\alpha\in V\cup F}$
to satisfy 
\begin{equation}
\bar{m}-\frac{1}{s_{\alpha}}\Big[\underset{\alpha_{2}\in V\cup F}{\overset{\phantom{\alpha_{2}VF}}{\sum}}\mathbf{z}_{\alpha_{2}}\Big]_{\alpha}^{\phantom{\alpha}}=x_{\alpha}\text{ for all }\alpha\in V\cup E\cup\{r\}\text{ such that }s_{\alpha}>0.\label{eq:x-alpha-relation}
\end{equation}
We now explain \eqref{eq:x-alpha-relation} further. Algorithm \ref{alg:undir-alg}
starts with $s_{\alpha}^{0}=1$ if $\alpha\in V$, and zero otherwise,
and $y_{\alpha}^{0}$ is such that $\frac{1}{|V|}\sum_{i\in V}y_{i}^{0}=\bar{m}$.
So a possible choice of $y_{i}^{0}$ is $\bar{x}_{i}$, as defined
in \eqref{eq:dist_opt_pblm}. Recall $\{\mathbf{z}_{\alpha}^{0}\}_{\alpha\in F}$
are to be defined to satisfy Proposition \ref{prop:sparsity}(2),
and that as long as $\frac{1}{|V|}\sum_{i\in V}y_{i}^{0}=\bar{m}$,
$\{\mathbf{z}_{\alpha}^{0}\}_{\alpha\in F}$ can be chosen to satisfy
\eqref{eq:x-alpha-relation}. It is clear to see that operations $A$
and $B$ can be written as a composition of operations $D$ and $E$.
In Theorem \ref{thm:recurrence}, we shall prove that throughout Algorithm
\ref{alg:undir-alg}, $\{\mathbf{z}_{\alpha}\}_{\alpha\in F}$ can
be chosen so that \eqref{eq:x-alpha-relation} is satisfied. 

For now, the only new part in Algorithm \ref{alg:Op_ABC} compared
to \cite{Bof_Carli_Schenato_2017} is operation $C$. Let the tuple
$T$ be defined as 
\begin{equation}
T=(\{s_{\alpha}\}_{\alpha\in V\cup E\cup\{r\}},\{y_{\alpha}\}_{\alpha\in V\cup E\cup\{r\}},\{x_{\alpha}\}_{\alpha\in V\cup E\cup\{r\}},\{\mathbf{z}_{\alpha}\}_{\alpha\in V\cup F}),\label{eq:tuple}
\end{equation}
and define $T^{k}$ similarly. Define $\Val(T)$ as 
\[
\begin{array}{c}
\Val(T):=\underset{i\in V}{\sum}\mathbf{f}_{i}^{*}(\mathbf{z}_{i})+\underset{\beta\in F}{\sum}\delta_{H_{\beta}^{\perp}}(\mathbf{z}_{\beta})+\underset{\alpha\in V\cup E}{\overset{\phantom{\alpha V}}{\sum}}\frac{s_{\alpha}}{2}\left\Vert x_{\alpha}\right\Vert ^{2}.\end{array}
\]
\begin{algorithm}[h]
\begin{lyxalgorithm}
\label{alg:undir-alg} (Main algorithm) We have the following algorithm. 

$\quad$Start with $y_{\alpha}^{0}$ such that $\sum_{i\in V}y_{i}^{0}=|V|\bar{m}$,
and $y_{\alpha}^{0}=0$ for all $\alpha\in E\cup\{r\}$. 

$\quad$Start with $s_{\alpha}^{0}$ such that $s_{i}=1$ for all
$i\in V$ and $s_{\alpha}^{0}=0$ for all $\alpha\in E\cup\{r\}$.

$\quad$Start with $\mathbf{z}_{i}=0$ for all $i\in V$. 

$\quad$Start with $\sigma_{i,y}^{0}=0$ and $\sigma_{i,s}^{0}=0$
for all $i\in V$.

$\quad$Start with $\rho_{(i,j),y}^{0}=0$ and $\rho_{(i,j),s}^{0}=0$
for all $(i,j)\in E$. 

$\quad$For $k=1,\dots$

$\quad$$\quad$\% Carry data from last iteration.

$\quad$$\quad$$y_{\alpha}^{k}=y_{\alpha}^{k-1}$ and $s_{\alpha}^{k}=s_{\alpha}^{k-1}$
for all $\alpha\in V\cup E$ 

$\quad$$\quad$$\sigma_{i,y}^{k}=\sigma_{i,y}^{k-1}$, $\sigma_{i,s}^{k}=\sigma_{i,s}^{k-1}$
and $[\mathbf{z}_{i}^{k}]_{i}=[\mathbf{z}_{i}^{k-1}]_{i}$ for all
$i\in V$

$\quad$$\quad$$\rho_{(i,j),y}^{k}=\rho_{(i,j),y}^{k-1}$ and $\rho_{(i,j),s}^{k}=\rho_{(i,j),s}^{k-1}$
for all $(i,j)\in E$

$\quad$$\quad$Perform one of operation A, B or C in Algorithm \ref{alg:Op_ABC}.

$\quad$end for
\begin{lyxalgorithm}
\label{alg:Op_ABC}(Operations $A$, $B$ and $C$) We describe operations
$A$, $B$ and $C$:

01$\quad$\textbf{$A$ (Node $i$ sends data) }

02$\quad$$\quad$Choose a node $i\in V$.

03$\quad$$\quad$$y_{i}^{k}=y_{i}^{k}/(\outdeg(i)+1)$; $s_{i}^{k}:=s_{i}^{k}/(\outdeg(i)+1)$

04$\quad$$\quad$$\sigma_{i,y}^{k}=\sigma_{i,y}^{k}+y_{i}^{k}$;
$\sigma_{i,s}^{k}=\sigma_{i,s}^{k}+s_{i}^{k}$.

05$\quad$\textbf{$B$ (Node $j$ receives data from $i$) }

06$\quad$$\quad$Choose edge $(i,j)\in E$ so that $j$ receives
data along $(i,j)$.

07$\quad$$\quad$$y_{j}^{k}=y_{j}^{k}+\sigma_{i,y}^{k}-\rho_{(i,j),y}^{k}$;
$s_{j}^{k}=s_{j}^{k}+\sigma_{i,s}^{k}-\rho_{(i,j),s}^{k}$

08$\quad$$\quad$$\rho_{(i,j),y}^{k}=\sigma_{i,y}^{k}$; $\rho_{(i,j),s}^{k}=\sigma_{i,s}^{k}$

09$\quad$\textbf{$C$ (Update $y_{j}$ and $[\mathbf{z}_{j}]_{j}$
by minimizing dual function)}

10$\quad$$\quad$Choose a node $j\in V$. 

11$\quad$$\quad$$x_{temp}=\frac{1}{s_{j}^{k}}(y_{j}^{k}+[\mathbf{z}_{j}^{k}]_{j})$

12$\quad$$\quad$$[\mathbf{z}_{j}^{k}]_{j}:=\underset{z}{\arg\min}\frac{s_{j}^{k}}{2}\|x_{temp}-\frac{1}{s_{j}^{k}}z\|^{2}+f_{j}^{*}(z)$

13$\quad$$\quad$$y_{j}^{k}=s_{j}^{k}x_{temp}-[\mathbf{z}_{j}^{k}]_{j}$
\begin{lyxalgorithm}
\label{alg:opns-D-E}(Operations $D$ and $E$) We describe operations
$D$ and $E$. The inputs are $\{s_{\alpha}^{\circ}\}_{\alpha\in V\cup E\cup\{r\}}$
and $\{y_{\alpha}^{\circ}\}_{\alpha\in V\cup E\cup\{r\}}$, and the
outputs are $\{s_{\alpha}^{+}\}_{\alpha\in V\cup E\cup\{r\}}$ and
$\{y_{\alpha}^{+}\}_{\alpha\in V\cup E\cup\{r\}}$. The $\mathbf{z}_{i}$
values remain unchanged for all $i\in V$.

14$\quad$\textbf{$D$ (Split with $r$) Suppose $s_{r}^{\circ}=0$.}

15$\quad$$\quad$Choose $\bar{\alpha}\in V\cup E$.

16$\quad$$\quad$Choose $s_{\bar{\alpha}}^{+}$ and $s_{r}^{+}$
to be such that $s_{\bar{\alpha}}^{+}+s_{r}^{+}=s_{\bar{\alpha}}^{\circ}$

17$\quad$$\quad$Let $y_{\bar{\alpha}}^{+}=\frac{s_{\bar{\alpha}}^{+}}{s_{\bar{\alpha}}^{\circ}}y_{\bar{\alpha}}^{\circ}$
and $y_{r}^{+}=\frac{s_{r}^{+}}{s_{\bar{\alpha}}^{\circ}}y_{\bar{\alpha}}^{\circ}$.

18$\quad$$\quad$$s_{\alpha}^{+}=s_{\alpha}^{\circ}$ and $y_{\alpha}^{+}=y_{\alpha}^{\circ}$
for all $\alpha\notin\{r,\bar{\alpha}\}$, and $[\mathbf{z}_{i}^{+}]_{i}=[\mathbf{z}_{i}^{\circ}]_{i}$
for all $i\in V$.

19$\quad$\textbf{$E$ (Combine with $r$) Suppose $s_{r}^{\circ}>0$.}

20$\quad$$\quad$Choose $\bar{\alpha}_{2}\in V\cup E$.

21$\quad$$\quad$Let $s_{\bar{\alpha}_{2}}^{+}=s_{\bar{\alpha}_{2}}^{\circ}+s_{r}^{\circ}$
and $s_{r}^{+}=0$.

22$\quad$$\quad$Let $y_{\bar{\alpha}_{2}}^{+}=y_{\bar{\alpha}_{2}}^{\circ}+y_{r}^{\circ}$
and $y_{r}^{+}=0$.

23$\quad$$\quad$$s_{\alpha}^{+}=s_{\alpha}^{\circ}$ and $y_{\alpha}^{+}=y_{\alpha}^{\circ}$
for all $\alpha\notin\{r,\bar{\alpha}_{2}\}$, and $[\mathbf{z}_{i}^{+}]_{i}=[\mathbf{z}_{i}^{\circ}]_{i}$
for all $i\in V$.
\end{lyxalgorithm}

\end{lyxalgorithm}

\end{lyxalgorithm}

\end{algorithm}
It is clear from the minimization step in line 12 that if $T^{k+1}$
is obtained from $T^{k}$ using operation $C$, then $\Val(T^{k+1})\leq\Val(T^{k})$.
Through duality and \eqref{eq:y-s-x-x}, the problem of finding new
$[\mathbf{z}_{j}^{k}]_{j}$ and $y_{j}^{k}$ in lines 12 and 13 can
be rewritten to solve a primal problem instead: 

12$^{\prime}\quad$$x_{j}^{k}=\underset{x}{\arg\min}\frac{s_{j}^{k}}{2}\|x_{temp}-x\|^{2}+f_{j}(x)$

13$^{\prime}\quad$$[\mathbf{z}_{j}^{k}]_{j}=s_{j}^{k}(x_{temp}-x_{j}^{k})$

In Section \ref{sec:Convergence-analysis}, we shall analyze operations
$D$ and $E$ in order to draw conclusions about Algorithm \ref{alg:undir-alg}.

\begin{figure}[h]
\begin{tabular}{|ccccc|}
\hline 
\begin{tabular}{c}
\includegraphics[scale=0.25]{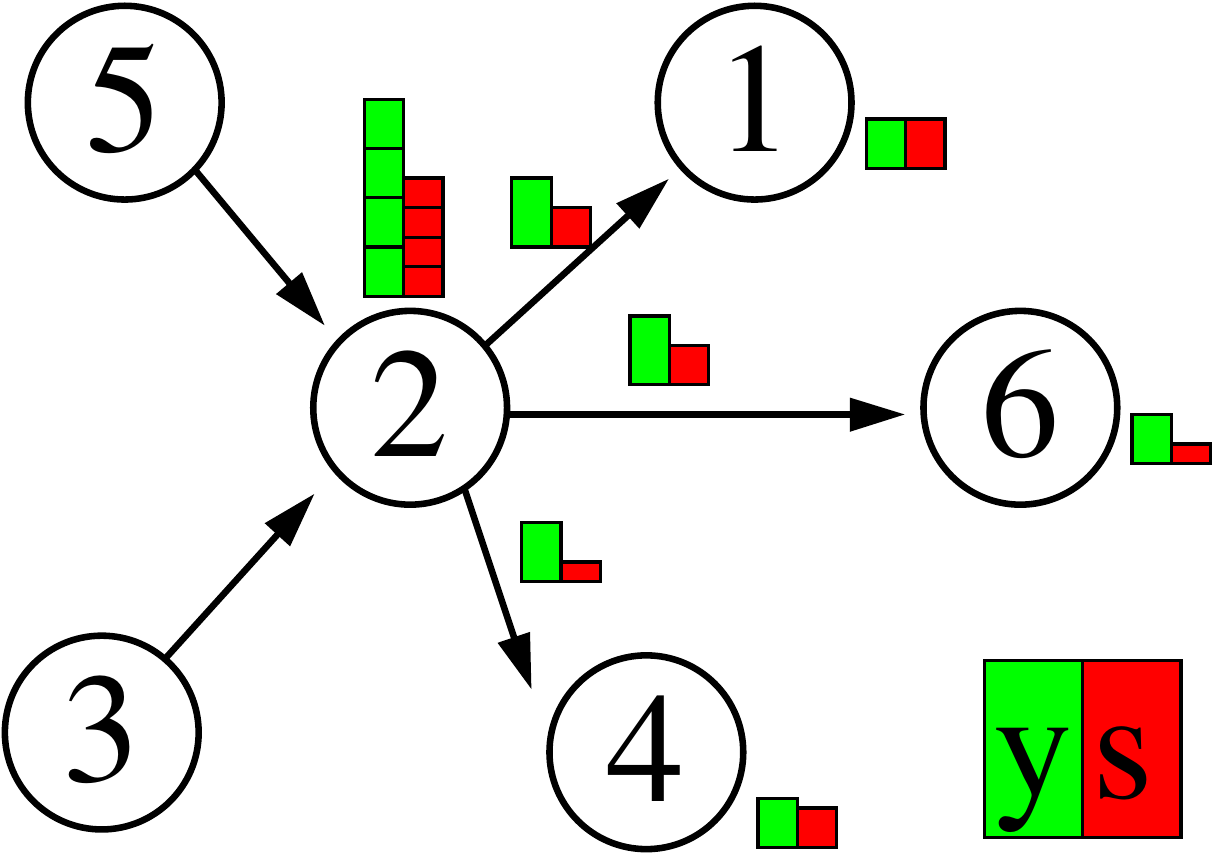}\tabularnewline
\end{tabular} & $\!\!\!\!\!\!\!\!\!\!$%
\begin{tabular}{c}
$\xrightarrow{\scriptsize{\text{Oper. A}}}$\tabularnewline
\end{tabular}$\!\!\!\!\!\!\!\!\!\!\!\!\!\!\!\!\!\!\!$ & %
\begin{tabular}{c}
\includegraphics[scale=0.25]{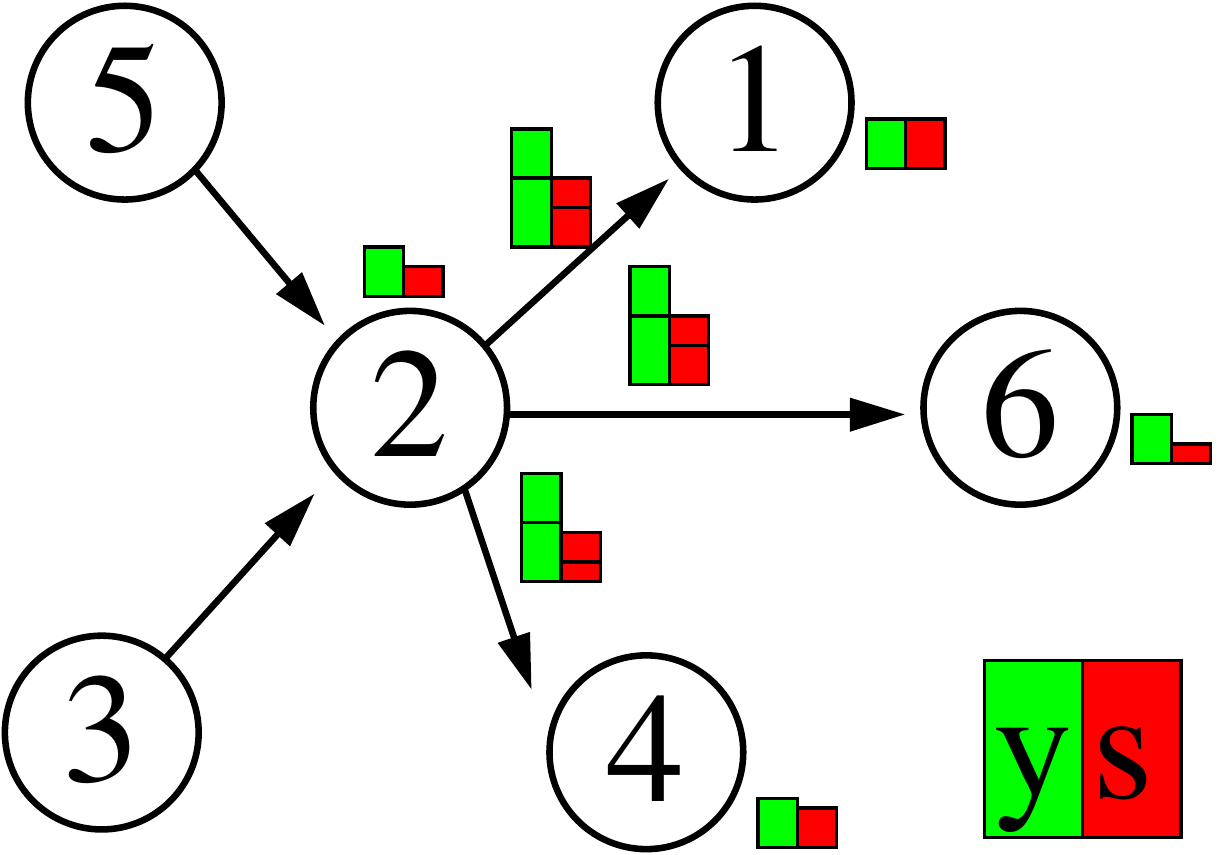}\tabularnewline
\end{tabular} & $\!\!\!\!\!\!\!\!\!$%
\begin{tabular}{c}
$\xrightarrow{\scriptsize{\text{Oper. B}}}$\tabularnewline
\end{tabular}$\!\!\!\!\!\!\!\!\!\!\!\!\!\!\!\!\!\!\!$ & %
\begin{tabular}{c}
\includegraphics[scale=0.25]{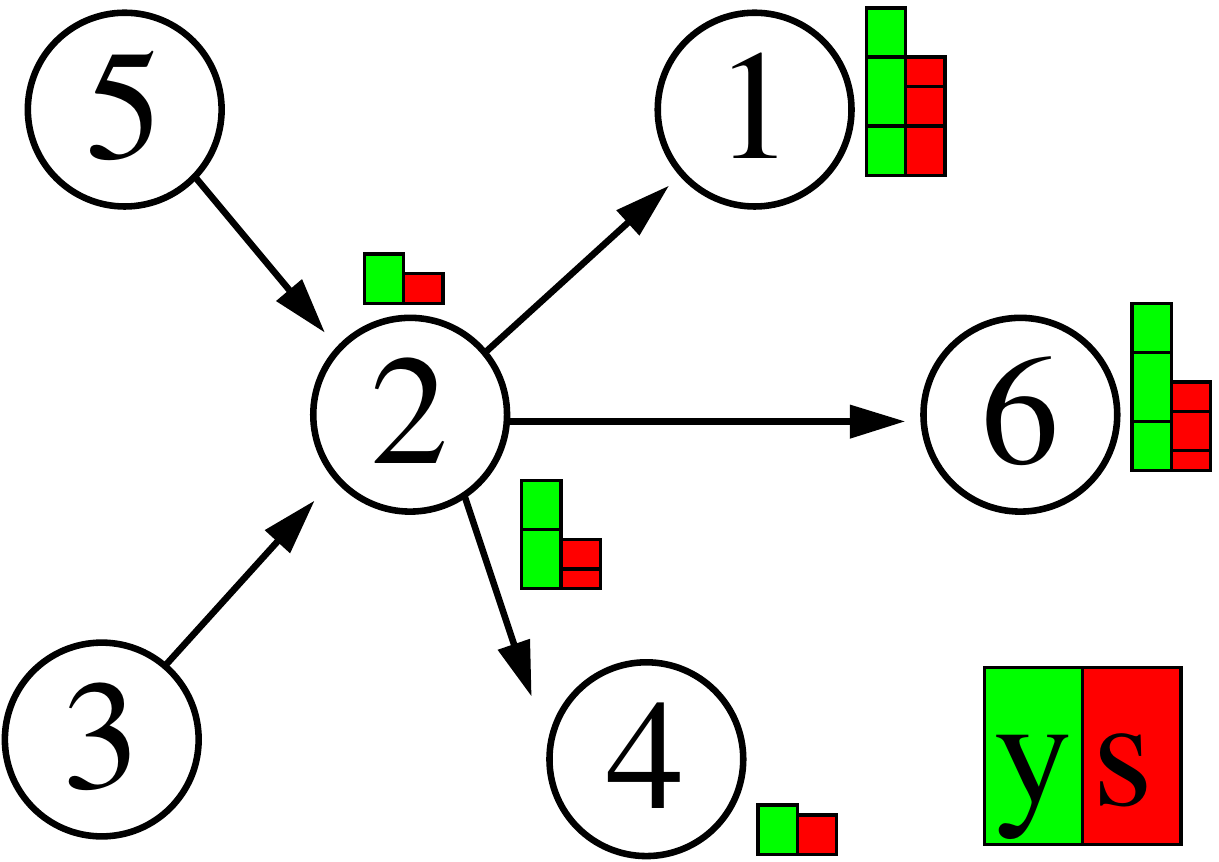}\tabularnewline
\end{tabular}\tabularnewline
\hline 
\multicolumn{5}{|c|}{%
\begin{tabular}{c}
\includegraphics[scale=0.22]{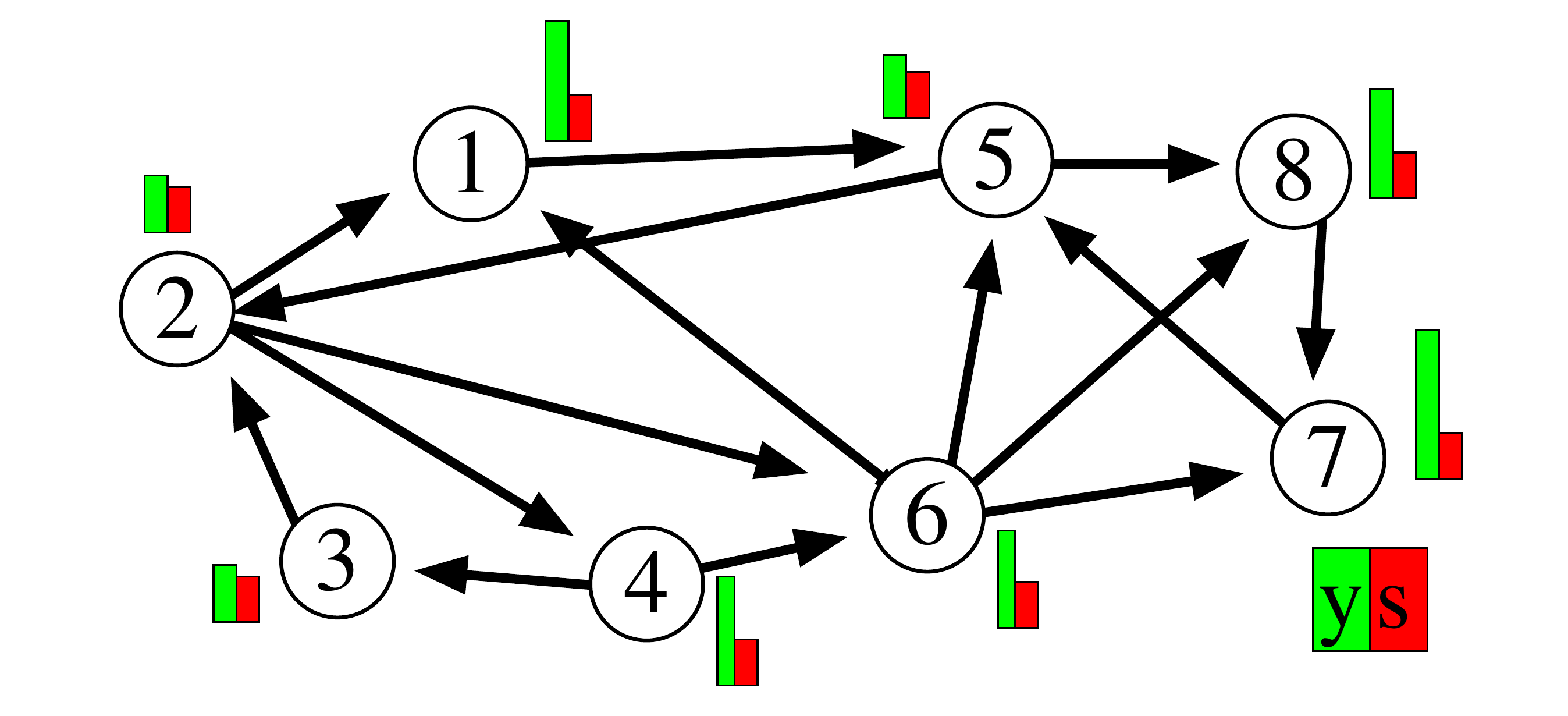}\tabularnewline
\end{tabular}$\!\!\!\!\!\!\!\!\to\!\!\!\!\!\!\!\!\!$%
\begin{tabular}{c}
\includegraphics[scale=0.22]{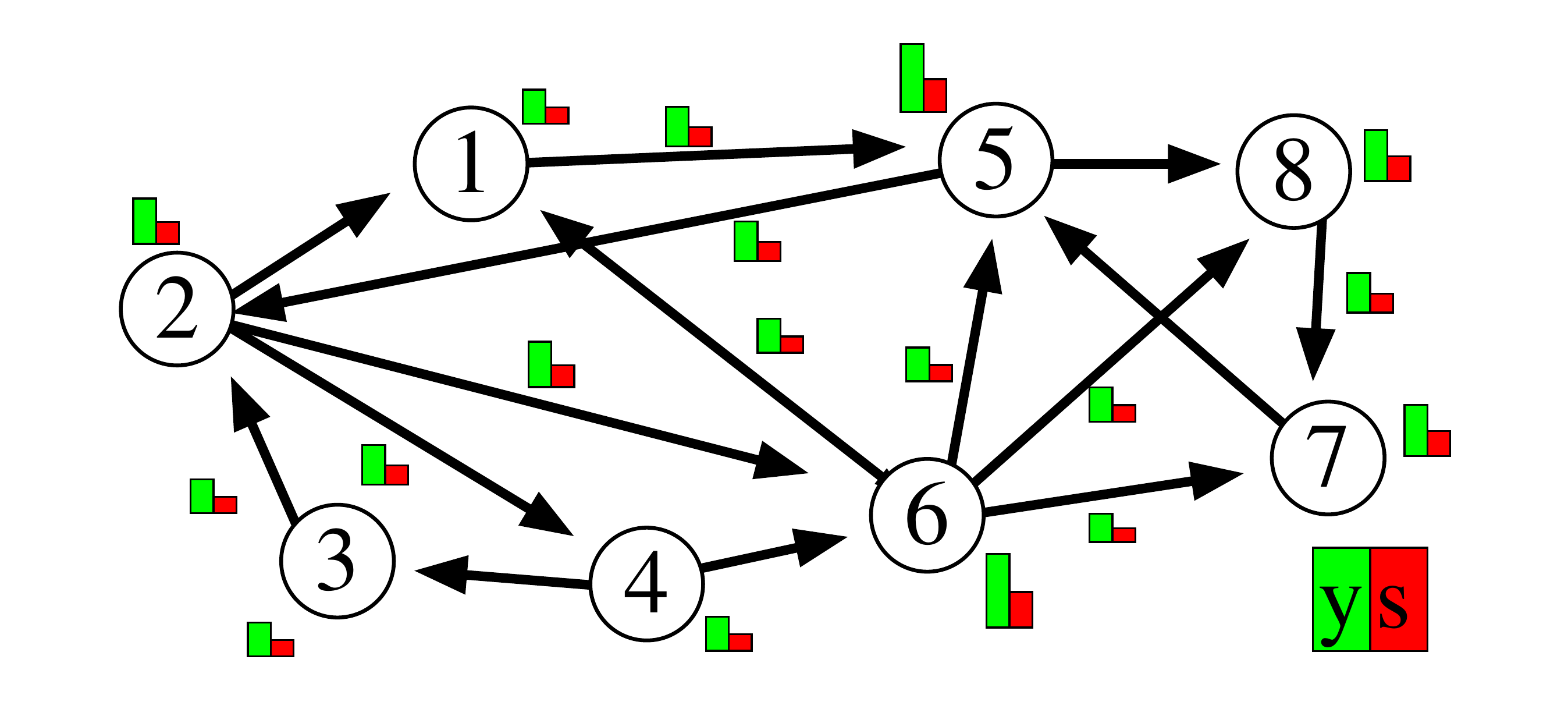}\tabularnewline
\end{tabular}}\tabularnewline
\hline 
\end{tabular}

\caption{\label{fig:The-fig}The top diagram illustrates Operations $A$ and
$B$ in Algorithm \ref{alg:Op_ABC} due to \cite{Bof_Carli_Schenato_2017}.
See Remark \ref{rem:on-opn-B}. The bottom diagram illustrates that
in \cite{Bof_Carli_Schenato_2017}, after many iterations, the value
$y_{\alpha}/s_{\alpha}\protect\overset{\eqref{eq:y-s-x-x}}{=}x_{\alpha}$
converges to the desired average $\frac{1}{|V|}\sum_{i\in V}\bar{x}_{i}$
for all $\alpha\in V\cup E$. }
\end{figure}

\begin{rem}
\label{rem:on-opn-B}In Figure \ref{fig:The-fig}, we show a case
where the data sent by node 2 along edge (2,4) has not yet been received
by node $4$. As explained in \cite{Bof_Carli_Schenato_2017} and
the relevant background works, this represents information that is
delayed and not lost.
\end{rem}

\section{\label{sec:Convergence-analysis}Convergence analysis}

In this section, we prove the convergence of Algorithm \ref{alg:undir-alg}.
 We show that operations $D$ and $E$ result in a nonincreasing
dual objective value $\Val(\cdot)$, and that they preserve the relations
\eqref{eq:x-alpha-relation} and \eqref{eq_m:y-s-x}. Since operation
$C$ is already easily seen to result in a nonincreasing dual objective
value, we would then see that $\{\Val(T^{k})\}_{k}$ is a nonincreasing
sequence. We then end by showing that under reasonable conditions,
$\{\Val(T^{k})\}_{k}$ converges to the minimum value of \eqref{eq:dual-2}.
This, together with strong duality implied from the distributed Dykstra's
algorithm and \eqref{eq:big-formula} show the convergence of all
the $\{x_{i}\}_{i\in V}$ to the primal minimizer. 

\subsection{\label{subsec:D-and-E}Operations $D$ and $E$}

First, we recall the operations $A$, $B$ and $C$ in Algorithm \ref{alg:Op_ABC}
and Operations $D$ and $E$ in Algorithm \ref{alg:opns-D-E}. It
is clear that operations $A$ and $B$ can be written as the composition
of a finite number of operations $D$ and $E$. 

\begin{thm}
\label{thm:recurrence} (Operations $D$ and $E$) Consider the following
conditions. 
\begin{itemize}
\item [(A)]The tuple $T$ defined in \eqref{eq:tuple} satisfies \eqref{eq:x-alpha-relation}
and \eqref{eq:y-s-x-x}.
\item [(B)]$\mathbf{z}_{\{r,\alpha\}}^{\circ}=0$ for all $\alpha\in V\cup E$,
$s_{r}^{\circ}=0$ and $y_{r}^{\circ}=0$.
\item [(C)]There is some $\bar{\alpha}_{1}\in V\cup E$ such that $\mathbf{z}_{\{r,\alpha\}}^{\circ}=0$
for all $\alpha\neq\bar{\alpha}_{1}$. 
\end{itemize}
Then the following hold. 
\begin{enumerate}
\item Suppose condition (A) is satisfied for the tuple 
\begin{equation}
T^{\circ}=(\{s_{\alpha}^{\circ}\}_{\alpha\in V\cup E\cup\{r\}},\{y_{\alpha}^{\circ}\}_{\alpha\in V\cup E\cup\{r\}},\{x_{\alpha}^{\circ}\}_{\alpha\in V\cup E\cup\{r\}},\{\mathbf{z}_{\alpha}^{\circ}\}_{\alpha\in V\cup F})\label{eq:tuple-1}
\end{equation}
and condition (B) is satisfied at the start of operation $D$ in Algorithm
\ref{alg:opns-D-E}. Then we can find $\{\mathbf{z}_{\alpha}^{+}\}_{\alpha\in V\cup F}$
such that the tuple $T^{+}$ defined in a similar manner to \eqref{eq:tuple-1}
satisfies conditions (A) and (C). Moreover, $\Val(T^{+})=\Val(T^{\circ})$. 
\item Suppose condition (A) is satisfied for the tuple $T^{\circ}$, condition
(C) is satisfied at the start of operation $E$, and $\{\bar{\alpha}_{1},\bar{\alpha}_{2}\}\in F$.
Then we can find $\{\mathbf{z}_{\alpha}^{+}\}_{\alpha\in V\cup F}$
such that the tuple $T^{+}$ satisfies conditions (A) and (B). Moreover,
$\Val(T^{+})\leq\Val(T^{\circ})$.
\end{enumerate}
\end{thm}

\begin{proof}
 We assume that \eqref{eq:y-s-x-x} holds throughout. To simplify
notations in the proof, all sums ``$\sum_{\beta}$'' in the proof
are of the form ``$\sum_{\beta\in V\cup F}$''. We first look at
Operation $D$.  Let $\mathbf{z}_{\{r,\bar{\alpha}\}}^{+}$ be such
that 
\begin{equation}
\begin{array}{c}
[\mathbf{z}_{\{r,\bar{\alpha}\}}^{+}]_{\alpha}=0\text{ for all }\alpha\text{\ensuremath{\notin}}\{r,\bar{\alpha}\}\text{, }[\mathbf{z}_{\{r,\bar{\alpha}\}}^{+}]_{r}=\frac{s_{r}^{+}}{s_{\bar{\alpha}}^{\circ}}\Big[\underset{\beta}{\sum}\mathbf{z}_{\beta}^{\circ}\Big]_{\bar{\alpha}}\text{, and }[\mathbf{z}_{\{r,\bar{\alpha}\}}^{+}]_{\bar{\alpha}}=-[\mathbf{z}_{\{r,\bar{\alpha}\}}^{+}]_{r},\end{array}\label{eq:def-z-1}
\end{equation}
and let all other $\mathbf{z}_{\alpha}^{+}$ be equal to $\mathbf{z}_{\alpha}^{\circ}$.
Since condition (B) holds for $T^{\circ}$ and \eqref{eq:def-z-1}
holds, condition (C) holds for $T^{+}$. So we only need to check
that \eqref{eq:x-alpha-relation} holds for $x_{\bar{\alpha}}^{+}$
and $x_{r}^{+}$. Note that 
\begin{equation}
\begin{array}{c}
\frac{1}{s_{\bar{\alpha}}^{+}}\Big[\underset{\beta}{\sum}\mathbf{z}_{\beta}^{+}\Big]_{\bar{\alpha}}\overset{\scriptsize{\text{\eqref{eq:def-z-1}}}}{=}\frac{1}{s_{\bar{\alpha}}^{+}}\left(\Big[\underset{\beta}{\overset{\phantom{\beta}}{\sum}}\mathbf{z}_{\beta}^{\circ}\Big]_{\bar{\alpha}}-\frac{s_{r}^{+}}{s_{\bar{\alpha}}^{\circ}}\Big[\underset{\beta}{\sum}\mathbf{z}_{\beta}^{\circ}\Big]_{\bar{\alpha}}\right)\overset{\scriptsize{\text{Line 16}}}{=}\frac{1}{s_{\bar{\alpha}}^{\circ}}\Big[\underset{\beta}{\sum}\mathbf{z}_{\beta}^{\circ}\Big]_{\bar{\alpha}}.\end{array}\label{eq:chain-1}
\end{equation}
(In the first equation for \eqref{eq:chain-1}, note that $\mathbf{z}_{\beta}^{+}=\mathbf{z}_{\beta}^{\circ}$
if $\beta\notin\{r,\bar{\alpha}\}$.) Since $\frac{s_{r}^{+}}{s_{\bar{\alpha}}^{\circ}}[\sum_{\beta}\mathbf{z}_{\beta}^{\circ}]_{\bar{\alpha}}\overset{\eqref{eq:def-z-1}}{=}[\mathbf{z}_{\{r,\bar{\alpha}\}}^{+}]_{r}\overset{\scriptsize{\text{Remark \ref{rem:On-index-r}}}}{=}[\sum_{\beta}\mathbf{z}_{\beta}^{+}]_{r}$,
we have 
\begin{equation}
\begin{array}{c}
\bar{m}-\frac{1}{s_{\bar{\alpha}}^{+}}\Big[\underset{\beta}{\sum}\mathbf{z}_{\beta}^{+}\Big]_{\bar{\alpha}}\overset{\eqref{eq:chain-1}}{=}\bar{m}-\frac{1}{s_{\bar{\alpha}}^{\circ}}\Big[\underset{\beta}{\overset{\phantom{\beta}}{\sum}}\mathbf{z}_{\beta}^{\circ}\Big]_{\bar{\alpha}}=\bar{m}-\frac{1}{s_{r}^{+}}\Big[\underset{\beta}{\sum}\mathbf{z}_{\beta}^{+}\Big]_{r}.\end{array}\label{eq:chain-2}
\end{equation}
So $x_{\bar{\alpha}}^{+}\overset{\scriptsize{\eqref{eq:y-s-x-x}}}{=}\frac{y_{\bar{\alpha}}^{+}}{s_{\bar{\alpha}}^{+}}\overset{\scriptsize{\text{Line 17}}}{=}\frac{y_{\bar{\alpha}}^{\circ}}{s_{\bar{\alpha}}^{\circ}}\overset{\scriptsize{\eqref{eq:y-s-x-x},\eqref{eq:x-alpha-relation}}}{=}\bar{m}-\frac{1}{s_{\bar{\alpha}}^{\circ}}\left[\sum_{\beta}\mathbf{z}_{\beta}^{\circ}\right]_{\bar{\alpha}}\overset{\scriptsize{\eqref{eq:chain-2}}}{=}\bar{m}-\frac{1}{s_{\bar{\alpha}}^{+}}\left[\sum_{\beta}\mathbf{z}_{\beta}^{+}\right]_{\bar{\alpha}}$,
which means \eqref{eq:x-alpha-relation} holds for $x_{\bar{\alpha}}^{+}$.
Similarly, \eqref{eq:x-alpha-relation} holds for $x_{r}^{+}$. In
fact, \eqref{eq:chain-2} also gives $x_{\bar{\alpha}}^{+}=x_{\bar{\alpha}}^{\circ}=x_{r}^{+}$,
which gives 
\[
\begin{array}{c}
\frac{s_{\bar{\alpha}}^{+}}{2}\left\Vert x_{\bar{\alpha}}^{+}\right\Vert ^{2}+\frac{s_{r}^{+}}{2}\left\Vert x_{r}^{+}\right\Vert ^{2}\overset{\scriptsize{\text{\eqref{eq:chain-2},\eqref{eq:y-s-x-x}}}}{=}\frac{s_{\bar{\alpha}}^{+}+s_{r}^{+}}{2}\left\Vert x_{\bar{\alpha}}^{\circ}\right\Vert ^{2}\overset{\scriptsize{\text{Line 16}}}{=}\frac{s_{\bar{\alpha}}^{\circ}}{2}\left\Vert x_{\bar{\alpha}}^{\circ}\right\Vert ^{2}.\end{array}
\]
 This in turn means $\Val(T^{+})=\Val(T^{\circ})$. 

We now look at operation $E$. By Proposition \ref{prop:sparsity}(2),
let $v\in\mathbb{R}^{m}$ be such that 
\begin{equation}
[\mathbf{z}_{\{r,\bar{\alpha}_{1}\}}^{\circ}]_{\alpha}=0\text{ for all }\alpha\text{\ensuremath{\notin}}\{r,\bar{\alpha}_{1}\},[\mathbf{z}_{\{r,\bar{\alpha}_{1}\}}^{\circ}]_{r}=v\text{, and }[\mathbf{z}_{\{r,\bar{\alpha}_{1}\}}^{\circ}]_{\bar{\alpha}_{1}}=-v.\label{eq:def-z-2}
\end{equation}
We then construct $\mathbf{z}_{\{\bar{\alpha}_{1},\bar{\alpha}_{2}\}}^{+}$
by 
\begin{equation}
[\mathbf{z}_{\{\bar{\alpha}_{1},\bar{\alpha}_{2}\}}^{+}]_{\alpha}=0\text{ for all }\alpha\text{\ensuremath{\notin}}\{\ensuremath{\bar{\alpha}_{1}},\bar{\alpha}_{2}\},[\mathbf{z}_{\{\bar{\alpha}_{1},\bar{\alpha}_{2}\}}^{+}]_{\bar{\alpha}_{1}}=-v\text{, and }[\mathbf{z}_{\{\bar{\alpha}_{1},\bar{\alpha}_{2}\}}^{+}]_{\bar{\alpha}_{2}}=v,\label{eq:def-z-3}
\end{equation}
and $\mathbf{z}_{\{r,\bar{\alpha}_{1}\}}^{+}=0$. All other $\mathbf{z}_{\alpha}^{+}$
shall be equal to $\mathbf{z}_{\alpha}^{\circ}$. Since condition
(C) holds for $T^{\circ}$, condition (B) holds for $T^{+}$. We now
check \eqref{eq:x-alpha-relation} for $x_{\bar{\alpha}_{2}}^{+}$.
We have
\begin{equation}
\begin{array}{c}
\Big[\underset{\beta}{\overset{\phantom{\beta}}{\sum}}\mathbf{z}_{\beta}^{\circ}\Big]_{\bar{\alpha}_{2}}+\Big[\underset{\beta}{\overset{\phantom{\beta}}{\sum}}\mathbf{z}_{\beta}^{\circ}\Big]_{r}\overset{\eqref{eq:def-z-2}}{=}\Big[\underset{\beta}{\overset{\phantom{\beta}}{\sum}}\mathbf{z}_{\beta}^{\circ}\Big]_{\bar{\alpha}_{2}}+v\overset{\eqref{eq:def-z-3}}{=}\Big[\underset{\beta}{\overset{\phantom{\beta}}{\sum}}\mathbf{z}_{\beta}^{+}\Big]_{\bar{\alpha}_{2}}.\end{array}\label{eq:alpha-bar-2-formula}
\end{equation}
Hence 
\begin{eqnarray*}
 &  & \begin{array}{c}
y_{\bar{\alpha}_{2}}^{+}\overset{\scriptsize{\text{Line 22}}}{=}y_{\bar{\alpha}_{2}}^{\circ}+y_{r}^{\circ}\end{array}\\
 & \overset{\eqref{eq:y-s-x-x},\eqref{eq:x-alpha-relation}}{=} & \begin{array}{c}
s_{\bar{\alpha}_{2}}^{\circ}\Big(\bar{m}-\frac{1}{s_{\bar{\alpha}_{2}}^{\circ}}\Big[\underset{\beta}{\overset{\phantom{\beta}}{\sum}}\mathbf{z}_{\beta}^{\circ}\Big]_{\bar{\alpha}_{2}}\Big)+s_{r}^{\circ}\Big(\bar{m}-\frac{1}{s_{r}^{\circ}}\Big[\underset{\beta}{\sum}\mathbf{z}_{\beta}^{\circ}\Big]_{r}\Big)\end{array}\\
 & \overset{\scriptsize{\text{Line 21,\eqref{eq:alpha-bar-2-formula}}}}{=} & \begin{array}{c}
s_{\bar{\alpha}_{2}}^{+}\Big(\bar{m}-\frac{1}{s_{\bar{\alpha}_{2}}^{+}}\Big[\underset{\beta}{\sum}\mathbf{z}_{\beta}^{+}\Big]_{\bar{\alpha}_{2}}\Big),\end{array}
\end{eqnarray*}
which, through \eqref{eq:y-s-x-x}, shows that \eqref{eq:x-alpha-relation}
holds for $x_{\bar{\alpha}_{2}}^{+}$. From the convexity of the norm-squared
function $\|\cdot\|^{2}$, we have 
\begin{equation}
\begin{array}{c}
\frac{s_{\bar{\alpha}_{2}}^{\circ}}{s_{\bar{\alpha}_{2}}^{\circ}+s_{r}^{\circ}}\|x_{\bar{\alpha}_{2}}^{\circ}\|^{2}+\frac{s_{r}^{\circ}}{s_{\bar{\alpha}_{2}}^{\circ}+s_{r}^{\circ}}\|x_{r}^{\circ}\|^{2}\geq\left\Vert \frac{s_{\bar{\alpha}_{2}}^{\circ}x_{\bar{\alpha}_{2}}^{\circ}+s_{r}^{\circ}x_{r}^{\circ}}{s_{\bar{\alpha}_{2}}^{\circ}+s_{r}^{\circ}}\right\Vert ^{2}\overset{\scriptsize{\text{Lines 21,22},\eqref{eq:x-alpha-relation}}}{=}\left\Vert \frac{y_{\bar{\alpha}_{2}}^{+}}{s_{\bar{\alpha}_{2}}^{+}}\right\Vert ^{2}\overset{\eqref{eq:x-alpha-relation}}{=}\left\Vert x_{\bar{\alpha}_{2}}^{+}\right\Vert ^{2}.\end{array}\label{eq:use-identity}
\end{equation}
The above inequality shows that $\Val(T^{+})\leq\Val(T^{\circ})$. 
\end{proof}

\subsection{Convergence result}

In this subsection, we prove our convergence result. 

Let $x^{*}$ be the optimal solution for \eqref{eq:dist_opt_pblm},
and $\mathbf{x}^{*}=\{x_{\alpha}^{*}\}_{\alpha\in V\cup E\cup\{r\}}$
be the optimal solution for \eqref{eq:2nd-primal}. It is clear that
$x_{\alpha}^{*}=x^{*}$ for all $\alpha\in V\cup E\cup\{r\}$. We
prove the boundedness of $\{x_{\alpha}^{k}\}_{k}$ for all $\alpha\in V\cup E\cup\{r\}$. 
\begin{thm}
\label{thm:bddness}(Boundedness of $\{x_{\alpha}\}$) Let $x^{*}$
be the optimal solution for \eqref{eq:dist_opt_pblm}. Suppose Algorithm
\ref{alg:undir-alg} is such that there is some $\bar{\epsilon}>0$
such that $s_{i}>\bar{\epsilon}$ for all $i\in V$. Then the iterates
$\{x_{\alpha}^{k}\}_{k}$ are bounded for all $\alpha\in V\cup E\cup\{r\}$. 
\end{thm}

\begin{proof}
Recall $\mathbf{x}^{*}$ defined just before the statement of this
result. From Fenchel duality, we have 
\begin{equation}
\mathbf{f}_{i}(\mathbf{x}^{*})+\mathbf{f}_{i}^{*}(\mathbf{z}_{i})\geq\langle\mathbf{x}^{*},\mathbf{z}_{i}\rangle\text{ and }\delta_{H_{\beta}}(\mathbf{x}^{*})+\delta_{H_{\beta}^{\perp}}(\mathbf{z}_{\beta})\geq0.\label{eq:Fenchel-ineq}
\end{equation}
Let $v_{\alpha}^{k}:=\frac{1}{s_{\alpha}^{k}}\left[\sum_{\alpha_{2}\in V\cup F}\mathbf{z}_{\alpha_{2}}^{k}\right]_{\alpha}$.
Using a technique in \cite[(8)]{Gaffke_Mathar}, the duality gap (i.e.,
the difference between the objective values of \eqref{eq:2nd-primal}
and \eqref{eq:dual-1}) satisfies
\begin{eqnarray}
 &  & \begin{array}{c}
\underset{\alpha\in V\cup E\cup\{r\}}{\overset{\phantom{\alpha\in V\cup E\cup\{r\}}}{\sum}}\frac{s_{\alpha}^{k}}{2}\|x^{*}-\bar{m}\|^{2}+\underset{i\in V}{\sum}\mathbf{f}_{i}(\mathbf{x}^{*})+\underset{\beta\in F}{\sum}\delta_{H_{\beta}}(\mathbf{x}^{*})-\frac{|V|}{2}\|\bar{m}\|^{2}\end{array}\nonumber \\
 &  & \begin{array}{c}
+\underset{i\in V}{\sum}\mathbf{f}_{i}^{*}(\mathbf{z}_{i}^{k})+\underset{\beta\in F}{\sum}\delta_{H_{\beta}^{\perp}}(\mathbf{z}_{\beta}^{k})+\underset{\alpha\in V\cup E\cup\{r\}}{\overset{\phantom{\alpha\in V\cup E\cup\{r\}}}{\sum}}\frac{s_{\alpha}^{k}}{2}\left\Vert \bar{m}-v_{\alpha}^{k}\right\Vert ^{2}\end{array}\nonumber \\
 & \overset{\eqref{eq:Fenchel-ineq},\eqref{eq:condn-on-s}}{\geq} & \begin{array}{c}
\left\langle \mathbf{x}^{*},\underset{\alpha\in V\cup E\cup\{r\}}{\overset{\phantom{\alpha\in V\cup E\cup\{r\}}}{\sum}}\mathbf{z}_{\alpha}^{k}\right\rangle +\underset{\alpha\in V\cup E\cup\{r\}}{\overset{\phantom{\alpha\in V\cup E\cup\{r\}}}{\sum}}s_{\alpha}^{k}\left(\frac{1}{2}\|x^{*}-\bar{m}\|^{2}+\frac{1}{2}\left\Vert \bar{m}-v_{\alpha}^{k}\right\Vert ^{2}-\frac{1}{2}\|\bar{m}\|^{2}\right)\end{array}\nonumber \\
 & = & \begin{array}{c}
\underset{\alpha\in V\cup E\cup\{r\}}{\overset{\phantom{\alpha\in V\cup E\cup\{r\}}}{\sum}}s_{\alpha}^{k}\left(\left\langle x^{*},v_{\alpha}^{k}\right\rangle +\frac{1}{2}\|x^{*}\|^{2}-\left\langle x^{*},\bar{m}\right\rangle +\frac{1}{2}\left\Vert \bar{m}-v_{\alpha}^{k}\right\Vert ^{2}\right)\end{array}\nonumber \\
 & = & \begin{array}{c}
\underset{\alpha\in V\cup E\cup\{r\}}{\overset{\phantom{\alpha\in V\cup E\cup\{r\}}}{\sum}}\frac{s_{\alpha}^{k}}{2}\left\Vert x^{*}-\left(\bar{m}-v_{\alpha}^{k}\right)\right\Vert ^{2}\overset{\eqref{eq:x-alpha-relation}}{=}\underset{\alpha\in V\cup E\cup\{r\}}{\overset{\phantom{\alpha\in V\cup E\cup\{r\}}}{\sum}}\frac{s_{\alpha}^{k}}{2}\left\Vert x^{*}-x_{\alpha}^{k}\right\Vert ^{2}.\end{array}\label{eq:big-formula}
\end{eqnarray}
The formula in the first two lines in \eqref{eq:big-formula} is nonincreasing
due to Theorem \ref{thm:recurrence}. Suppose that when $k=0$, the
value in the first two lines in the formula in \eqref{eq:big-formula}
is $C_{0}$. Since $s_{i}^{k}>\bar{\epsilon}$ for all $i\in V$,
we have $\|x^{*}-x_{i}^{k}\|\overset{\eqref{eq:big-formula}}{<}\sqrt{C_{0}/\bar{\epsilon}}$
for all $i\in V$ and $k$. From operation $A$, we can see that either
$s_{e}^{k}=0$ or $s_{e}^{k}>\bar{\epsilon}/|V|$, so $\|x^{*}-x_{e}^{k}\|<\sqrt{C_{0}|V|/\bar{\epsilon}}$
for all $e\in E$ and $k$. So $\{x_{\alpha}^{k}\}_{k}$ is bounded
for all $\alpha\in V\cup E$. A similar analysis shows the same for
$\{x_{r}^{k}\}_{k}$. (If we only use operations $A$, $B$ and $C$,
$s_{r}^{k}$ and $y_{r}^{k}$ remain zero throughout.)
\end{proof}
\begin{thm}
\label{thm:convergence}(Convergence to dual objective value) Suppose
there is some number $\bar{\epsilon}>0$ such that $s_{i}^{k}>\bar{\epsilon}$
for all $k\geq0$ and $i\in V$. Assume that the iterates $\{[\mathbf{z}_{i}^{k}]_{i}\}_{i\in V}$
of Algorithm \ref{alg:undir-alg} are bounded. Suppose that there
is a number $K$ such that in $K$ consecutive iterations,
\begin{enumerate}
\item Operations $A$ and $C$ are carried out for all nodes $i\in V$ separately
at least once, and
\item Operation $B$ is carried out for all edges $(i,j)\in E$ separately
at least once. 
\end{enumerate}
Then $\{\Val(T^{k})\}_{k}$ is nonincreasing, and its limit is the
dual objective value of \eqref{eq:dual-2}. 
\end{thm}

\begin{proof}
We consider the tuple $T^{k}:=(\{s_{\alpha}^{k}\}_{\alpha\in V\cup E},\{x_{\alpha}^{k}\}_{\alpha\in V\cup E},\{[\mathbf{z}_{i}^{k}]_{i}\}_{i\in V})$.
Since all these quantities are bounded by Theorem \ref{thm:bddness}
and the assumptions, there is a subsequence $\{T^{k_{i}}\}_{i}$ such
that $\lim_{i\to\infty}T^{k_{i}}$ exists. Taking subsequences if
necessary, we can assume that $\lim_{i\to\infty}(T^{k_{i}},T^{k_{i}+1})$
exists, and that the operation (either $A$, $B$ or $C$) to get
$T^{k_{i}+1}$ from $T^{k_{i}}$ are all the same. Applying this procedure
repeatedly shows that we can assume that 
\[
\lim_{i\to\infty}(T^{k_{i}},T^{k_{i}+1},\dots,T^{k_{i}+K})
\]
exists. For each $j\in\{1,\dots,K\}$, the operations to get $T^{k_{i}+j}$
from $T^{k_{i}+j-1}$ are indepedent on $i$. For $j\in\{0,\dots,K\}$,
let the limits $\lim_{i\to\infty}T^{k_{i}+j}$ be 
\[
\tilde{T}_{j}:=(\{\tilde{s}_{\alpha}^{j}\}_{\alpha\in V\cup E},\{\tilde{x}_{\alpha}^{j}\}_{\alpha\in V\cup E},\{[\tilde{\mathbf{z}}_{i}^{j}]_{i}\}_{i\in V}).
\]
From the continuity of the operations $A$, $B$ and $C$, the operations
to get $\tilde{T}_{j}$ from $\tilde{T}_{j-1}$ must be the same as
that of getting $T^{k_{i}+j}$ from $T^{k_{i}+j-1}$. Since $\{\Val(T^{k})\}_{k}$
is a nonincreasing sequence, we conclude that 
\begin{equation}
\Val(\tilde{T}_{0})=\Val(\tilde{T}_{1})=\cdots=\Val(\tilde{T}_{K}).\label{eq:values-equal}
\end{equation}
Since $s_{i}^{k}>\bar{\epsilon}$ for all $k\geq0$ and $i\in V$,
we have $\tilde{s}_{i}^{j}>\bar{\epsilon}$ for all $i\in V$ and
$j\in\{0,\dots,T\}$. 
\begin{claim}
\label{claim:the-claim}There is some $\tilde{x}^{*}\in\mathbb{R}^{m}$
such that $\tilde{x}_{\alpha}^{j}=\tilde{x}^{*}$ for all $\alpha\in V\cup E$
and $j\in\{0,\dots,K\}$. 
\end{claim}

We now prove Claim \ref{claim:the-claim}. Recall that the operations
$A$, $B$ and $C$ are continuous. Seeking a contradiction, suppose
$\tilde{x}_{\alpha}^{j-1}\neq\tilde{x}_{\alpha}^{j}$ for some $j\in\{1,\dots,K\}$.
There are three cases. 

\textbf{Case 1:} Operation $C$ was used to get $\tilde{T}_{j}$ from
$\tilde{T}_{j-1}$.

Recall that operation $C$ gives $\tilde{x}_{\alpha}^{j-1}=\tilde{x}_{\alpha}^{j}$
when $\alpha\in E$, so $\alpha$ has to be in $V$. Hence $\tilde{s}_{\alpha}^{j-1}>\bar{\epsilon}$.
We recall that 
\[
\begin{array}{c}
\tilde{x}_{\alpha}^{j}\overset{\scriptsize{\text{Line 12}^{\prime}}}{=}\underset{x}{\arg\min}\frac{\tilde{s}_{\alpha}^{j-1}}{2}\left\Vert \tilde{x}_{\alpha}^{j-1}+\frac{1}{\tilde{s}_{\alpha}^{j-1}}z_{\alpha}^{j-1}-x\right\Vert ^{2}+f_{\alpha}(x)\text{ for all }\alpha\in V.\end{array}
\]
Since $\tilde{x}_{\alpha}^{j}\neq\tilde{x}_{\alpha}^{j-1}$, we have
$\tilde{T}_{j}\leq\tilde{T}_{j-1}-\frac{\tilde{s}_{\alpha}^{j-1}}{2}\|\tilde{x}_{\alpha}^{j}-\tilde{x}_{\alpha}^{j-1}\|^{2}<\tilde{T}_{j-1}$,
which contradicts \eqref{eq:values-equal}. 

\textbf{Case 2:} Operation $B$ was used to get $\tilde{T}_{j}$ from
$\tilde{T}_{j-1}$.

By looking further at operation $B$, we deduce that $\alpha=j'$
for some $j'\in V$, and operation $B$ was performed on $(i',j')\in E$.
Since Operation $B$ makes use of Operations $D$ and $E$, and so
we look at how Operation $E$ would lead to a contradiction if $\tilde{x}_{\alpha}^{j-1}\neq\tilde{x}_{\alpha}^{j}$.
The inequality in \eqref{eq:use-identity} is strict if $x_{\bar{\alpha}_{2}}^{\circ}\neq x_{r}^{\circ}$.
Translating this observation to the case of Operation $B$ shows that
$\tilde{T}_{j}\leq\tilde{T}_{j-1}-\frac{s_{\alpha}^{j-1}}{2}\|\tilde{x}_{\alpha}^{j}-\tilde{x}_{\alpha}^{j-1}\|^{2}$,
which again contradicts \eqref{eq:values-equal}. The analysis here
also shows that $\tilde{x}_{j'}^{j}=\tilde{x}_{(i',j')}^{j}$. 

\textbf{Case 3:} Operation $A$ was used to get $\tilde{T}_{j}$ from
$\tilde{T}_{j-1}$.

Suppose the node chosen in operation $A$ was $i'$. The techniques
used in this case is similar to that in case 2. So we just summarize
the conclusions, which are that $\tilde{x}_{i'}^{j}=\tilde{x}_{i'}^{j-1}$,
$\tilde{x}_{i'}^{j}=\tilde{x}_{(i',j')}^{j}$ and $\tilde{x}_{(i',j')}^{j-1}=\tilde{x}_{(i',j')}^{j}$
for all $i'\in V$ and all out-neighbors $j'$ of $i'$. 

Since the graph $G$ was assumed to be strongly connected, the analysis
in all three cases shows the conclusion of Claim \ref{claim:the-claim}.
$\hfill\triangle$

We now show that for all $i\in V$, $[\tilde{\mathbf{z}}_{i}^{j}]_{i}$
equal to some $\tilde{z}_{i}^{*}$ such that $\tilde{z}_{i}^{*}\in\partial f_{i}(x^{*})$
for all $j\in\{0,\dots,K\}$. Seeking a contradiction, suppose $[\tilde{\mathbf{z}}_{i}^{j-1}]_{i}\neq[\tilde{\mathbf{z}}_{i}^{j}]_{i}$.
The only possibility is that operation $C$ was used to get from $\tilde{T}_{j-1}$
to $\tilde{T}_{j}$. Since operation $C$ is continuous, we have 
\begin{equation}
\begin{array}{c}
[\tilde{\mathbf{z}}_{i}^{j}]_{i}\overset{\scriptsize{\text{Line 12}}}{=}\underset{z}{\arg\min}\frac{\tilde{s}_{i}^{j-1}}{2}\left\Vert \tilde{x}^{*}+\frac{1}{\tilde{s}_{i}^{j-1}}([\tilde{\mathbf{z}}_{i}^{j}]_{i}-z)\right\Vert ^{2}+f_{i}^{*}(z),\end{array}\label{eq:z-prox-form}
\end{equation}
If $[\tilde{\mathbf{z}}_{i}^{j-1}]_{i}\neq[\tilde{\mathbf{z}}_{i}^{j}]_{i}$,
then we have $\tilde{T}_{j}\leq\tilde{T}_{j-1}-\frac{\tilde{s}_{i}^{j-1}}{2}\|\frac{1}{\tilde{s}_{i}^{j-1}}([\tilde{\mathbf{z}}_{i}^{j-1}]_{i}-[\tilde{\mathbf{z}}_{i}^{j}]_{i})\|^{2}<\tilde{T}_{j-1}$,
which contradicts \eqref{eq:values-equal}. The formula \eqref{eq:z-prox-form}
also shows that $\tilde{z}_{i}=[\tilde{\mathbf{z}}_{i}^{j}]_{i}\in\partial f_{i}(\tilde{x}^{*})$. 

By making use of Operations $D$ and $E$, all the $\tilde{T}_{j}$
can be transformed to some $\hat{T}$ where $\hat{s}_{i}=1$ for all
$i\in V$, and $\hat{s}_{e}=0$ for all $e\in E$. From the discussion
in Subsection \ref{subsec:D-and-E}, we can find some $\{\hat{\mathbf{z}}_{\beta}\}_{\beta\in V\cup F}$
such that $\Val(\hat{T})$ is also equal to $F(\{\hat{\mathbf{z}}_{\alpha}\}_{\alpha\in V\cup F})$,
where 
\[
\begin{array}{c}
F(\{\mathbf{z}_{\alpha}\}_{\alpha\in V\cup F}):=\underset{i\in V}{\sum}\mathbf{f}_{i}^{*}(\mathbf{z}_{i})+\underset{\beta\in F}{\sum}\delta_{H_{\beta}^{\perp}}(\mathbf{z}_{\beta})+\underset{\alpha\in V}{\sum}\frac{1}{2}\left\Vert \bar{m}-\left[\underset{\alpha_{2}\in V\cup F}{\sum}\mathbf{z}_{\alpha_{2}}\right]_{\alpha}\right\Vert ^{2}.\end{array}
\]
 We see that $F(\cdot)$ is the sum of separable terms (the first
two sums) and a smooth term (the last sum). In view of $\tilde{z}_{i}^{*}\in\partial f_{i}(\tilde{x}^{*})$
for all $i\in V$, the partial subdifferential of $F(\cdot)$ with
respect to $\mathbf{z}_{i}$ is zero at $\{\hat{\mathbf{z}}_{\beta}\}_{\beta\in V\cup F}$.
Since all $\hat{x}_{\alpha}$ are equal for all $\alpha\in V\cup E$,
the partial subdifferential of $F(\cdot)$ with respect to $\mathbf{z}_{\gamma}$
is also zero at $\{\hat{\mathbf{z}}_{\beta}\}_{\beta\in V\cup F}$
for all $\gamma\in F$. By some basic theory on block coordinate minimization
for convex problems, we have $\{\hat{\mathbf{z}}_{\beta}\}_{\beta\in V\cup F}$
being a minimizer of $F(\cdot)$, which shows that $\tilde{T}_{j}$
are at their minimum values for all $j\in\{0,\dots,K\}$. 
\end{proof}
The analysis in \cite{Pang_Dist_Dyk} (which traces back to \cite{Gaffke_Mathar}
and earlier) implies that there is strong duality between the problems
\eqref{eq:2nd-primal} and \eqref{eq:dual-1}, so the $\tilde{x}^{*}$
in the proof of Theorem \ref{thm:convergence} is $x^{*}$, the minimizer
of \eqref{eq:dist_opt_pblm}, and the quantity $\sum_{\alpha\in V\cup E\cup\{r\}}\frac{s_{\alpha}}{2}\left\Vert x^{*}-x_{\alpha}\right\Vert ^{2}$
in \eqref{eq:big-formula} converges to zero. Hence by Theorems \ref{thm:bddness}
and \ref{thm:convergence}, $\lim_{k\to\infty}x_{i}^{k}$ exists and
equals $x^{*}$ for all $i\in V$. 

\section{Numerical experiments}

We conduct some simple experiments by looking at the case where $m=6$
and the graph has 6 nodes and contains two cycles, $1\to2\to3\to5\to1$
and $2\to4\to6\to2$. Let $\mathbf{e}$ be \texttt{ones(m,1)}. First,
we find $\{v_{i}\}_{i\in V}$ and $\bar{x}$ such that $\sum_{i\in V}v_{i}+|V|(\mathbf{e}-\bar{x})=0$.
We then find closed convex functions $f_{i}(\cdot)$ such that $v_{i}\in\partial f_{i}(\mathbf{e})$.
It is clear from the KKT conditions that $\mathbf{e}$ is the primal
optimum solution to \eqref{eq:dist_opt_pblm} if $\bar{x}_{i}=\bar{x}$
for all $i\in V$. 

We define $f_{i}(\cdot)$ as functions of the following type:
\begin{itemize}
\item [(F-S)]$f_{i}(x):=\frac{1}{2}x^{T}A_{i}x+b_{i}^{T}x+c_{i}$, where
$A_{i}$ is of the form $vv^{T}+rI$, where $v$ is generated by \texttt{rand(m,1)},
$r$ is generated by \texttt{rand(1)}. $b_{i}$ is chosen to be such
that $v_{i}=\nabla f(\mathbf{e})$, and $c_{i}=0$. 
\item [(F-NS)]$f_{i}(x):=\max\{f_{i,1}(x),f_{i,2}(x)\}$, where $f_{i,j}(x):=\frac{1}{2}x^{T}A_{i}x+b_{i,j}^{T}x+c_{i,j}$
for $j\in\{1,2\}$, $A_{i}$ is of the form $vv^{T}+rI$, where $v$
is generated by \texttt{rand(m,1)}, $r$ is generated by \texttt{rand(1)},
$b_{i,1}$ and $b_{i,2}$ are chosen such that $v_{i}=\frac{1}{2}[\nabla f_{i,1}(\mathbf{e})+\nabla f_{i,2}(\mathbf{e})]$
but $v_{i}$ is neither $\nabla f_{i,1}(\mathbf{e})$ nor $\nabla f_{i,2}(\mathbf{e})$,
and $c_{i,1}$ and $c_{i,2}$ are chosen such that $f_{i,1}(\mathbf{e})=f_{i,2}(\mathbf{e})$. 
\end{itemize}
Note the algorithms in \cite{Notarstefano_gang_Newton_2017,Scutari_ASY_SONATA_2018}
do not handle nonsmooth functions. Also, our algorithm does not require
one to choose parameters to be small enough in order to achieve convergence.

\begin{figure}
\includegraphics[scale=0.25]{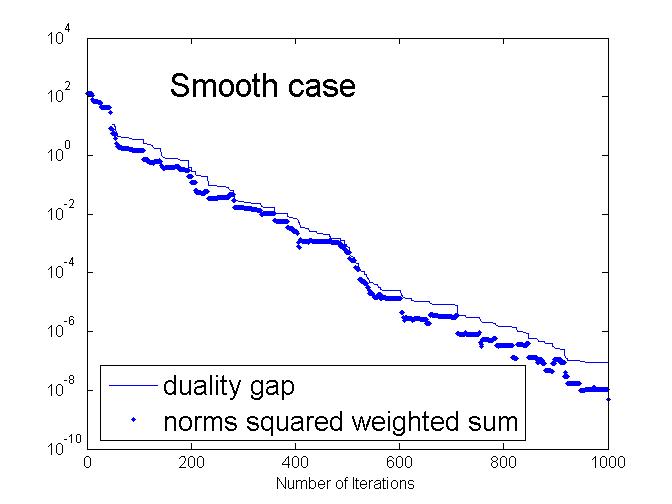}\includegraphics[scale=0.25]{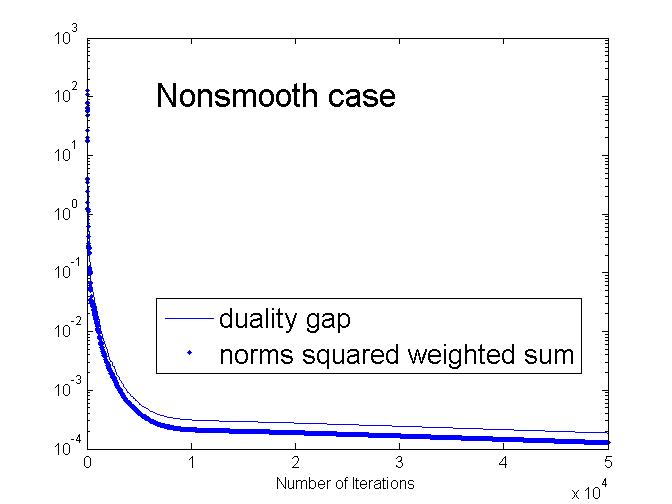}

\caption{\label{fig:fig}Plots of the formulas in \eqref{eq:big-formula}}

\end{figure}
We conduct two experiments, one when all functions are of the form
(F-S), and another when the functions are all of the form (F-NS).
The first and last formulas of \eqref{eq:big-formula} indicate how
fast the primal iterates $\{x_{\alpha}\}_{\alpha\in V\cup E}$ are
converging to the optimal solution $x^{*}$, and we call these values
the ``duality gap'' and the ``norms squared weighted sum''. Figure
\ref{fig:fig} shows a plot of the results obtained by a random experiment
where we perform 1000 iterations of the smooth case and 50000 iterations
of the nonsmooth case. The results observed are quite similar to that
in \cite{Pang_rate_D_Dyk}, where the edges are undirected. Specifically,
if all the functions $f_{i}(\cdot)$ are of the form (F-S), then we
observe linear convergence (though we have not proved this yet). If
all functions $f_{i}(\cdot)$ are of the form (F-NS), then we observe
sublinear convergence. Rather often, this sublinear convergence is
seen to be of order $O(1/k)$. 

\section{Conclusion}

To conclude, we make a few observations. The insight that the algorithm
in \cite{Bof_Carli_Schenato_2017} can be written as a dual ascent
optimization problem shows that other ideas in algorithm design that
were already laid out in the related papers \cite{Pang_Dist_Dyk,Pang_sub_Dyk,Pang_rate_D_Dyk,Pang_level_sets_Dyk}
on a distributed Dykstra's algorithm (for undirected edges), as well
as \cite{Pang_DBAP}, can be incorporated into the algorithm in this
paper. It is also straightforward to design an improved algorithm
for the case when some of the edges in the graph are undirected while
others are directed using Operations $D$ and $E$. We defer the proof
of linear convergence of the case when the functions $f_{i}(\cdot)$
in \eqref{eq:dist_opt_pblm} are smooth to a future paper. It would
be of interest to incorporate the algorithmic and theoretical properties
known for the case of undirected graphs to the case of directed graphs
with unreliable communications. 

\bibliographystyle{amsalpha}
\bibliography{../refs}

\end{document}